\documentclass[11pt]{article}
\usepackage{amsfonts}

\usepackage{graphics}
\usepackage{indentfirst}
\usepackage{cite}
\usepackage{latexsym}
\usepackage{amsmath,amsthm}
\usepackage{amssymb}
\usepackage[dvips]{epsfig}
\usepackage{amscd}
\usepackage{mathrsfs}

\usepackage{color}

\newtheorem{theorem}{Theorem}[section]
\newtheorem{remark}{Remark}[section]

\newtheorem{definition}{Definition}[section]
\newtheorem{lemma}[theorem]{Lemma}
\newtheorem{pro}{Proposition}[section]
\newtheorem{corollary}[theorem]{Corollary}

\newcommand{\n}{\rho}

\renewcommand{\div}{ {\rm div }  }

\newcommand{\na}{\nabla }
\newcommand{\vp}{\varphi }

\newcommand{\pa}{\partial}

\newcommand{\bt}{\begin{theorem}}
\newcommand{\bl}{\begin{lemma}}
\newcommand{\el}{\end{lemma}}
\newcommand{\et}{\end{theorem}}
\newcommand{\ga}{\gamma}

\newcommand{\te}{\theta}

\newcommand{\de}{\delta}

\newcommand{\la}{\label}

\newcommand{\om}{\Omega}
\newcommand{\ol}{\overline}

\newcommand{\bn}{\begin{eqnarray}}
\newcommand{\en}{\end{eqnarray}}
\newcommand{\bnn}{\begin{eqnarray*}}
\newcommand{\enn}{\end{eqnarray*}}

\newcommand{\bnnn}{\begin{eqnarray*}}
\newcommand{\ennn}{\end{eqnarray*}}
\newcommand{\ben}{\begin{enumerate}}
\newcommand{\een}{\end{enumerate}}

\newcommand{\ba}{\begin{aligned}}
\newcommand{\ea}{\end{aligned}}
\newcommand{\be}{\begin{equation}}
\newcommand{\ee}{\end{equation}}

\def\norm[#1]#2{\|#2\|_{#1}}

\def\lam{\lambda}
\def\ep{\varepsilon}

\makeatletter      % '@' is now a normal "letter" for TeX
\@addtoreset{equation}{section}
\makeatother       % '@' is restored as a "non-letter" character for TeX

\title{ Nonlinearly Exponential Stability
for Lions-Feireisl's Weak Solutions
to the  Barotropic Compressible Navier-Stokes Equations
with Large Potential External Forces
 \thanks{Partially supported by NSFC Grant Nos. 12361046, 12161019.}}

\date{}

\author{$\text{Lingping Kang}^{a,b}, \text{Yanfang Peng}^{c}, \text{Chengfeng Xiong}^{d,e}
\thanks{Email addresses: lingping@email.ncu.edu.cn(L.P. Kang), pyfang2005@sina.com(Y.F. Peng),\newline xiongchengfeng20@mails.ucas.ac.cn(C.F. Xiong)}$\\
	a. School of Mathematics and Computer Sciences,\\
	Nanchang University, Nanchang 330031, P. R. China;\\
	b. Institute of Mathematics and Interdisciplinary Sciences,\\
	Nanchang University, Nanchang 330031, P. R. China;\\
	c. School of Mathematical Sciences, \\Guizhou Normal University, Guiyang 550001, P. R. China;\\
    d. School of Mathematical Sciences,\\ University of Chinese Academy of Sciences,
     Beijing 100190, P. R. China;\\
    e. Institute of Applied Mathematics,\\ Academy of Mathematics and Systems Science (AMSS),\\
     Chinese Academy of Sciences, Beijing 100190, P.R. China.}

\begin{document}
\maketitle
\begin{abstract}
The large time behavior for  Lions-Feireisl's finite energy weak solutions to the  barotropic  compressible Navier-Stokes equations with  large potential external forces in  three-dimensional (3D) bounded domains is considered. Although the equilibrium state of density is not a constant anymore due to   the non-constant external forces,
by constructing a suitable Lyapunov functional and using the extra integrability of the density,
after expanding the  difference of the density
and its steady state in a Taylor series with respect to the  difference of some power function of  density
 and that of the steady density,
it is proved that any Lions-Feireisl's finite energy weak solution
would decay exponentially to the equilibrium state as time tends to infinity.
%  any Lions-Feireisl's finite energy weak solutions decaying exponentially to the equilibrium state is proved.

\end{abstract}


\section{Introduction and main results}
The three-dimensional barotropic  compressible Navier-Stokes equations  read as follows:
\be\la{1.1}
\begin{cases} \rho_t + \div(\rho u) = 0,\\
 (\rho u)_t + \div(\rho u\otimes u) -\mu\Delta u -(\mu + \lam)\nabla\div u+ \nabla P = \rho f.
\end{cases}
\ee
Here  $t\ge 0, x=(x_1,x_2,x_3)\in \Omega\subset\mathbb{R}^3$. Moreover, $\rho=\n(x,t)$ and $u=(u_1(x,t),u_2(x,t), u_3(x,t))$ represent, respectively, the density and velocity, and the pressure $P$ is given by
\be\la{1.2}
P(\rho) = R\rho^{\gamma},\quad R>0, \ga>1.
\ee
In the sequel, without loss of generality, we set $R=1$.

The constant viscosity coefficients $\mu$ and $\lam$ satisfy the physical hypothesis
\be\la{1.3}
\mu>0, \quad 2\mu+3\lam\geq 0.
\ee
The external force $f(x)$ is a gradient of a scalar potential, that is
\be\la{F}
f(x)=\na F(x).
\ee
We consider the  homogeneous Dirichlet boundary conditions
\be\la{1.5}
 u|_{\partial\Omega}=0,
 \ee
 and the initial conditions
 \be\la{1.4}
 \rho(x,0)=\rho_0, \quad \rho u(x,0)=m_0.
 \ee

There is a large amount of literature  on the well-posedness and large time behavior of compressible Navier-Stokes equations.
Now, we  list some related important  results in the multi-dimensional case.  On one hand, for the compressible Navier-Stokes equations \eqref{1.1} without external forces, when the density is strictly away from vacuum,
the local well-posedness of classical or strong solutions has been studied in \cite{Se,N,So,T}.
The existence of global classical solutions was firstly established  by  Matsumura-Nishida \cite{MN1}  for initial
    data close to a non-vacuum equilibrium in $H^3$. Later, for  discontinuous initial data, Hoff \cite{H1,H2} proved   the existence of global solutions and   introduced a new type of a priori estimate on the material derivative $\dot{u}$.  Recently, Huang-Li-Xin \cite{HLX}
     stated the global existence and uniqueness of classical solutions
    with smooth initial data that are of small energy
    but possibly large oscillations and  a constant far-field behavior that could be either vacuum or non-vacuum. As for   \eqref{1.1} with external forces,  most of work focus on   the case of  small external forces,  see \cite{MN2,MN3,ST} for details.
For \eqref{1.1} with large external forces, the case is more complicated. Li-Zhang-Zhao\cite{LZZ} proved that the Cauchy problem has a unique global strong solution with large oscillations and interior vacuum provided that the initial data are of small energy and the unique steady state is strictly away from vacuum.
   Then,  when the initial energy is suitably small,
    but the density allows large oscillations and contains vacuum states,  Cai-Huang-Shi\cite{CHS} showed the global existence of strong or classical solutions
    to the initial boundary value problem for \eqref{1.1}.  Moreover, they established the exponential decay in time to the equilibrium in certain Sobolev spaces for the global strong or classical solutions. Furthermore,  for the existence of weak solutions with large data in the three-dimensional case, the major breakthrough is due to  Lions \cite{L2},
where  the  existence of global weak solutions were established  under the condition of $\ga\geq \frac{9}{5}$,
    which was  further released to $\ga>\frac{3}{2}$ later by Feireisl-Novotn\'{y}-Petzeltov\'a\cite{FNP}.
Moreover, Jiang-Zhang \cite{JZ} demonstrated the existence of global weak solutions for any $\ga>1$, when the initial data possess certain symmetry properties.

On the other hand, regarding  the large time behavior of  solutions, Matsumura-Nishida \cite{MN1, MN2, MN3} proved that the existence of global solutions
near a constant equilibrium state for the Cauchy problem in $\mathbb{R}^3$, in bounded domains, in the half space and exterior domains, respectively.
  Moreover,  under different boundary conditions, Feireisl-Petzeltov\'a \cite{FP} and  Novotn\'{y}-Str\v{a}skraba \cite{NS}
    showed, respectively, that the density of global weak solutions converges to the steady state  in $L^p$ space for some $p > \frac{3}{2}$
    as time tends to infinity, provided that there exists a unique steady state.
Later, Padula \cite{P} demonstrated that in a bounded domain,
    the rest state is exponentially stable with respect to a large class of perturbations in the absence of vacuum.
Recently, Fang-Zi-Zhang \cite{FZZ} proved that
    any finite energy weak solutions to the problem (1.1)
    without external forces in a bounded domain decay exponentially to the equilibrium state in $L^2$ norm. It should be pointed out that in {\cite{FZZ,P},  the existence of weak solutions  heavily relies on the basic assumption that the density is bounded from above or below.
    More recently, without any additional bounds assumptions on the density,
 Peng-Shi-Wu \cite{PSW}  stated  that any finite energy weak solutions obtained by Lions \cite{L2}
    and Feireisl-Novotn\'{y}-Petzeltov\'a \cite{FNP} decay exponentially to the equilibrium state. However, whether Lions-Feireisl’s
    finite energy weak solutions
    to the initial boundary value problem for \eqref{1.1}
    with large external forces
    (whose existence was established by \cite{L2,FNP}) decay
    exponentially to the equilibrium state or not remains open.
In this paper, we will  give  an affirmative   answer to this problem.

Before going further, we denote  $(\rho_s, u_s)$ as the equilibrium state of (\ref{1.1}). It is easy to see that $(\rho_s, u_s)$ satisfies the following  problem
\begin{equation}
		\begin{cases}
			\div(\rho_s u_s)=0,  &x \in \Omega, \\
			-\mu\Delta u_s-(\lam+\mu)\nabla\div u_s+\nabla P(\rho_s)= \rho_s\nabla F, & x \in \Omega, \\
u_s=0, &x\in \partial \Omega,\\
\int_\Omega \rho dx=\int_\Omega \rho_s dx.
		\end{cases}
		\label{1.8}
	\end{equation}
 As indicated by Lemma \ref{lem1} in Section 2,  there exists a
unique  equilibrium solution $(\rho_s,0)$ to (\ref{1.8}), satisfying
\be\la{rho}
\rho_s\in W^{1,\infty}(\Omega),\quad \inf_{\Omega}\rho_s>0,
\ee
  under the condition that
\be\la{mas}
\int_\Omega\bigg(\frac{\gamma-1}{\gamma}\big(F-\inf_{\Omega}{F}\big)\bigg)^{\frac{1}{\gamma-1}}dx<\int_\Omega \rho_0dx,
\ee
and   $\rho_s$ satisfies
\be
\begin{cases}\label{1.9}
\nabla P(\rho_s)= \rho_s\nabla F,\\
 \int_\Omega \rho_s dx= \int_\Omega \rho_0 dx.
\end{cases}
\ee
Here, we should point out that for the condition \eqref{mas}, we will make further discussion in Section 4.

Now, we introduce the definition of  finite energy weak solutions.
\begin{definition} {\rm( \cite{FNP, L2}, Finite energy weak solutions)}
A pair of functions $(\rho,u)$ will be denoted
as a finite energy weak solution of the problems \eqref{1.1} on $(0,\infty)\times \Omega$,
if the following statements hold:

$\bullet$ $\rho\geq 0$, $\rho\in L^\infty(0,\infty;L^\gamma(\Omega))$, $u\in L^2(0,\infty;H_0^1(\Omega))$.

$\bullet$ Equations \eqref{1.1} are satisfied in $\mathcal{D}'((0,\infty)\times\Omega)$.
Moreover, $\eqref{1.1}_1$ holds in $\mathcal{D}'((0,\infty)\times\mathbb{R}^3)$ provided $\rho$, $u$ were prolonged to be zero on $\mathbb{R}^3 \backslash \Omega$.

$\bullet$ The energy
\be\ba
  E(t) := \int_\om\left(\frac{1}{2}\rho|u|^2 + \frac{1}{\ga-1}\rho^\ga-\rho F\right)dx \ea\ee
is locally integrable on $(0,\infty)$ and for any $0 \leq \vp
\in \mathcal{D}(0,\infty)$, it holds that
\be\ba\label{engyinq}
  -\int_0^\infty\vp_tE(t)dt +
  \int_0^\infty\vp\int_\om\bigg(\mu|\na u|^2+(\lambda+\mu)(\div u)^2\bigg)dxdt\leq0.
\ea\ee

$\bullet$ Equation $\eqref{1.1}_1$ holds in the sense of renormalized solutions; more specifically, the equation
\be\la{1.17}
b(\rho)_t+\div(b(\rho)u)+(b'(\rho)\rho-b(\rho))\div u=0
\ee
holds in $\mathcal{D}'((0,\infty)\times \Omega)$ for any $b\in C^1(\mathbb{R})$ such that
$b'(z)=0$ for all $z\in\mathbb{R}$ large enough, say, $|z|\geq M$, where the constant $M$ may vary for different functions $b$.
\end{definition}
% \begin{remark}\la{rmk1}
% A direct consequence of \eqref{engyinq} is for any $t \geq 0$,
% \be\label{E11}
%  E(t)+\int_0^t \int_\Omega [\mu |\nabla u|^2+(\lambda +\mu)(
% \div u)^2]dxds\leq E_0,
% \ee
% with
% $$
% E_0:=\int_\Omega \left(\frac{|m_0|^2}{2\rho_0}+\frac{1}{\gamma-1}\rho_0^\gamma-\rho_0F\right)dx.
% $$
% \end{remark}
% \begin{remark}
% \textcolor{red}{By \eqref{1.9}, we have for any $t\geq 0$,
% \be\ba
% \int \left(G(\n,\n_s)+\frac{1}{2}\n|u|^2\right)dx + \int_0^t\int_\Omega [\mu |\na u|^2+(\lambda +\mu)(\div u)^2]dxds
% \leq E_0,\ea\ee
% with
% \be\ba
% G(\n,\n_s) = \n \int_{\n_s}^\n \frac{h^\ga-\n_s^\ga}{h^2}dh,
% \ea\ee
% and
% \be\ba
%  E_0 = \int_\om \left(\frac{|m_0|^2}{2\n_0}+G(\n_0,\n_s)\right)dx.
% \ea\ee}
% \end{remark}
Next, we state the following result concerning the existence of finite energy weak solutions to
the problem \eqref{1.1}--\eqref{1.4} established
by Lions\cite{L2} and Feireisl-Novotn\'y-Petzeltov\'a\cite{FNP}.
\begin{lemma}   \la{lem3}
Assume that $\Omega\subset\mathbb{R}^3$  is a bounded domain of the class of $C^{2+k}$, $k>0$. Let $F \in W^{1,\infty}(\om)$,
$\gamma>\frac{3}{2}$ and the initial data $(\rho_0, m_0)$ satisfy
\[
0\leq\rho_0\in L^\gamma(\Omega),\ \frac{|m_0|^2}{\rho_0}\in L^1(\Omega),
\]
with $m_0=0$ almost everywhere on the set $\{x\in \Omega| \rho_0(x)=0\}$.
Then there exists a finite energy weak solution $(\rho, u)$
of the problem \eqref{1.1}--\eqref{1.4} satisfying, for almost everywhere $t>0$,
\be
\int_\Omega \rho(t)dx=\int_\Omega \rho_0dx,
\ee
and for any $T>0$,
\be
\rho\in L^{\gamma+\theta_0}(0,T;\Omega),
\ee
with some positive constant $\theta_0\leq \frac{2\gamma-3}{3}$.
\end{lemma}

 For an integrable function $f$,  we denote
\be\nonumber
\ol{f} := \frac{1}{|\om|}\int_{\om}f dx,
\ee
as the mean value of $f$ over $\Omega$.

We are in a position to state our main result as follows.
\begin{theorem}\la{t1}
Assume that the conditions of  Lemma \ref{lem3} and  \eqref{mas} hold.
Then, there exist positive constants $C_1$ depending  only on
$\om,\ \ga,\ \mu,\ \lam,\ F,\ E_0,\ \ol{\n_0}$
and  $T_0$ depending on
$\om,\ \ga,\ F$, $\n$ and $u$ such that $(\n,u)$,
the finite energy weak solution to \eqref{1.1}--\eqref{1.4},
whose existence is guaranteed by Lemma \ref{lem3},
satisfies the decay property
\be\label{expdec}
\int_\Omega\bigg(\frac{1}{2}\rho|u|^2+G(\rho,\rho_s)\bigg)dx
\leq 8E_0 e^{-C_1(t-T_0)}, \quad \mathrm{for \ almost\ every}\  t>T_0,
\ee
% where the definition of
% $G(\rho,\rho_s)$ is given in (\ref{g}).
with
\be\label{g}
G(\rho, \rho_s):=\rho\int^\rho_{\rho_s}\frac{h^\gamma-{\rho^\gamma_s}}{h^2}dh,
\ee
and \be\ba
E_0 = \int_\om \left(\frac{|m_0|^2}{2\n_0} + G(\n_0,\n_s)\right) dx.
\ea\ee
\end{theorem}
\begin{remark}\la{rmk4}
We  emphasize that in Theorem \ref{t1},  $T_0$ actually depends on the weak solution $(\n,u)$,
  due to the fact that  we need Lemma \ref{lem2} to close the estimate of
  $G(\n,\n_s)$ (see \eqref{x21}--\eqref{x22}).   However, the convergence in
  Lemma \ref{lem2} depends on the specific weak solution $(\n,u)$ itself.
 Thus, it is an interesting question whether $T_0$ can be  independent of the weak solution $(\n,u)$ or not.  This is left for the future.
\end{remark}
\begin{remark}\la{rmk2}
 It should be noted that  proof of Theorem \ref{t1} relies heavily on the property that
  the solution $(\rho_s, 0)$ to the system \eqref{1.8} satisfies the  $0<\rho_s\in W^{1,\infty}(\om)$.
  On  one hand, to address the largeness of $\na F$, we need to convert it into
  $G(\rho,\rho_s)$ with the aid of \eqref{equation}, which requires
   $\rho_s > 0$ to  validate the multiplication by $\rho_s^{-1}$.
  On the other hand,   in the context of finite energy weak solutions,
  $G(\rho,\rho_s)$ is only  an $L^1$ function with respect to $x$.
  Therefore, to match
  $\na \rho_s^{-1}$ with $G(\rho,\rho_s)$, the condition $0<\rho_s\in W^{1,\infty}(\om)$ is essential(see \eqref{equation}).
 Moreover, it would be of interest to investigate the case that $\rho_s \in W^{1,q}(\om)$ for some $q > 3$
  and $\rho_s$ contains vacuum, which would be left for the future.
\end{remark}

\begin{remark}\la{rmk3}
In our discussion, \eqref{mas} is indeed both a necessary and sufficient
  condition for the existence of equilibrium solution $(\rho_s,0)$ to \eqref{1.8}
  with $\displaystyle \inf_\om \rho_s > 0$. Further details will be left in  Section 4.
\end{remark}

We now make some comments on the analysis of this paper.
To establish Theorem \ref{t1}, with the aid of the energy inequality and the
conservation of the energy (see (\ref{x1})), the key issue is to
establish new decay estimates for $G(\rho,\rho_s)$ (see (\ref{g}) for the definition).
 Compared with \cite{PSW}, the essential difficulty lies in the fact that $\rho_s$ is no longer a constant
due to the large potential forces. To overcome the difficulty,
it is crucial to  reselect a  suitable test function replacing $\mathcal{B}(\rho^\te - \rho_s^\te - \overline{\rho^\te - \rho_s^\te})$ in \cite{PSW}.
Firstly,  to overcome the lack of smallness due to the large potential forces, motivated by \cite{HLX},  we multiply both sides of $(\ref{1.1})_2$ by $\rho_s^{-1}$ to get
\be\ba\nonumber
  \rho_s^{-1}(\na\rho^\ga - \rho\na F)
&=\rho_s^{-1}(\nabla(\rho^\gamma-\rho_s^\gamma)
  -\gamma(\rho-\rho_s)\rho_s^{\gamma-2}\nabla\rho_s) \\
&
  =\nabla(\rho_s^{-1}(\rho^\gamma-\rho_s^\gamma))
  -(\gamma-1)G(\rho,\rho_s)\nabla\rho^{-1}_s.
\ea\ee
In this way, the large potential external force $\na F$ can be transformed
into the term  $G(\rho,\rho_s)$ with smallness,
at the expense of incurring a multiplier $\na\rho_s^{-1}$. Secondly, if we only choose $\rho_s^{-1}\mathcal{B}(\rho^\te - \rho_s^\te - \overline{\rho^\te - \rho_s^\te})$ as the test function, then by standard arguments, we find that the crucial term
\be\label{*}
-\int_0^\infty\vp\int_\Omega \rho_s^{-1}(\rho^\ga - \rho_s^\ga)
 \overline{(\rho^\te - \rho_s^\te)}dxdt
\ee
is difficult to be estimated directly
since it will change its sign which is  non-negative
in  \cite{PSW} where they only consider the case
that $\rho_s\equiv {\it constant}$.   To estimate \eqref{*},  we observe  that by Taylor expansion with remainder in integral form,
\be\ba\label{texp}
  \rho - \rho_s = &\frac{1}{\te} \rho_s^{1-\te}(\rho^\te - \rho_s^\te)
  \\&+ \frac{1}{\te}(\frac{1}{\te}-1)(\rho^\te - \rho_s^\te)^2
    \int_0^1(1-\tau)\left(\rho_s^\te+\tau(\rho^\te - \rho^\te_s)\right)^{\frac{1}{\te} - 2}d\tau,
\ea\ee which combined with the simple fact that $\overline{\rho-\rho_s}=0$ implies that \bnn\ba
 &\overline{ \rho_s^{1-\te}(\rho^\te - \rho_s^\te)}
  \\&=- (\frac{1}{\te}-1)\overline{(\rho^\te - \rho_s^\te)^2
    \int_0^1(1-\tau)\left(\rho_s^\te+\tau(\rho^\te - \rho^\te_s)\right)^{\frac{1}{\te} - 2}d\tau},
\ea\enn which   can be well controlled by
the product of the potential and a term with smallness
(cf. (\ref{x16}) and (\ref{x18})). In view of the above key observation, we succeed in constructing  a test function
$$
\varphi(t)\rho_s^{-1}\mathcal{B}(\rho_s^{1-\theta}([\rho^\theta]_{\ep}-[\rho_s^\theta]_{\ep})
-\overline{\rho_s^{1-\theta}([\rho^\theta]_{\ep}-[\rho_s^\theta]_{\ep})})
$$
 for $(\ref{1.1})_2$,  where  the mollified functions $[\rho^\te]_\ep$ is used to overcome
the lack of  regularity for $\rho^\te$. Then, carrying out the standard estimates, we build up  the desired estimate \eqref{3.9}
which is vital to get the decay estimate of the finite energy weak solutions.
Finally, observing that $(\rho^\theta -\rho_s^{\theta})^2$ can be bounded by $G(\rho, \rho_s)$, provided that $\theta $ is suitably small (see  \eqref{x2}),
we  construct a suitable Lyapunov functional and then finish the proof of Theorem \ref{t1}. It is worth noticing that property of $\rho_s$ (see \eqref{rho}),  which is ensured  by the assumption \eqref{mas},  plays a major impact on our proof.  Thus,  we make an additional discussion on \eqref{mas} in Section 4.

This paper is organized as follows.
In Section 2, we list a series of preliminary lemmas
which are used  in later arguments.
In Section 3, we concentrate on deducing some necessary estimates
and finally accomplishing Theorem \ref{t1}.
In Section 4, we provide  further discussion on assumption \eqref{mas}.
Throughout the paper, unless explicitly stated otherwise,
we denote by $C$  a positive generic constant depending
only on $\om,\ga,\mu,\lam,F,\ol{\n_0}$ and $E_0$, and $C$ may vary
in different cases. Also set $C(\alpha)$ to emphasize that
$C$ depends on the parameter $\alpha$.

\section{Preliminaries}
In this section, we recall  some elementary lemmas which are useful in our later discussion.

First, for $\eta$ as the standard mollifier in $\mathbb{R}^3$ and $f\in L_{loc}^1(\mathbb{R}^3)$, we set
$$
\eta_\ep(\cdot):=\frac{1}{\ep^3}\eta_\ep\left(\frac{\cdot}{\ep}\right), \quad [f]_{\ep}:=\eta_{\ep}*f.
$$

We start with the following standard properties of mollification.
\begin{lemma}    \la{lem2.1}{\rm(\cite{AF})}
Let $f$ be a function which is defined on $\mathbb{R}^3$ and vanishes identically outside a domain $\Omega\subset\mathbb{R}^3$. Then

{\rm(i)} \,\,if $f\in L^p(\Omega)$ with $1\leq p<\infty$, then $[f]_\ep\in L^p(\Omega)$.

{\rm(ii)}\,\,$\|[f]_\ep\|_{L^p(\Omega)}\leq\|f\|_{L^p(\Omega)},\,\, \displaystyle\lim_{\ep\rightarrow 0^+}\|[f]_\ep-f\|_{L^p(\Omega)}=0.$
\end{lemma}

The  following commutator estimates, first derived by DiPerna-Lions \cite{DL}, will play a significant  role in   our further discussion.
\begin{lemma}  \la{lem2.3}{\rm(\cite{DL, F,L1})}
Let $\Omega\subset\mathbb{R}^3$ be a domain and $f\in L^p(\Omega)$,
 $v\in[W^{1,q}(\om)]^3$ be given functions with $1<p,q<\infty$ and $\frac{1}{p}+\frac{1}{q}\leq 1$.
 Then, for any compact subset $K\subset\Omega$,
\be\la{2-3}
\begin{cases}
\|[\div(fv)]_\ep-\div([f]_\ep v)\|_{L^r(K)}\leq C(K)\|f\|_{L^p(\Omega)}\|v\|_{W^{1,q}(\Omega)},\\
 \|[\div(fv)]_\ep-\div([f]_\ep v)\|_{L^r(K)}\rightarrow 0, \ \ as\ \ep\rightarrow 0,
\end{cases}
\ee
provided $\ep>0$ is small enough and $\frac{1}{r}=\frac{1}{p}+\frac{1}{q}$. In addition, if $\Omega=\mathbb{R}^3$, $K$ can be replaced by $\mathbb{R}^3$.
\end{lemma}

As a direct consequence of Lemma \ref{lem2.3}, the following result can be found in \cite{L2}.
\begin{corollary} \label{rensol}{\rm(\cite{L2})}
  Let $(\rho,u)$ be the solution of the problem \eqref{1.1}--\eqref{1.4}
  as in Lemma \ref{lem3}.
  Then, prolonging  $(\rho,u)$ to zero in $\mathbb{R}^3\backslash\Omega$,
  for $0<\theta\leq\frac{\gamma}{2}$ and any $[\alpha,\beta]\subset (0,\infty)$, we have
\be\la{2.3}\ba
  \pa_t([\rho^\te]_\ep)+ \div([\rho^\te]_\ep u)
  =(1-\te)[\rho^\te\div u]_\ep + r_\ep,
\ea\ee
almost everywhere on $[\alpha,\beta]$,
 where
 $$r_\ep:= \div([\rho^\te]_{\ep}u) - [\div(\rho^\te u)]_\ep.$$
 Moreover,
\be\ba\label{x35}
  r_\ep\to 0\quad \mathrm{in} \ L^2(\alpha,\beta;L^q(\om)),
\ea\ee
for any $q\in[1,\frac{2\ga}{\ga+2\te}]$.
\end{corollary}

We now introduce  the well-known Bogovskii operator,
which was first constructed by the Bogovskii \cite{B},  and state some related properties.
\begin{lemma}   \la{lem2.5}{\rm(\cite{B,FNP,G})}
Let $\Omega$ be a bounded Lipschitz domain in $\mathbb{R}^3$ and consider an auxiliary problem
\be \la{2.4}
\div v=f, \ \ v|_{\partial \Omega}=0.
\ee
Then,
there exists a linear operator
$\mathcal{B}=[\mathcal{B}_1,\mathcal{B}_2,\mathcal{B}_3]$
enjoying the properties:

$\bullet$ $\mathcal{B}$ is a bounded linear operator from $\{f\in L^p(\Omega)|\int_{\Omega}fdx=0\}$ into $[W^{1,p}_0(\Omega)]^3$, that is,
\be\ba\label{bog1}
\|\mathcal{B}[f]\|_{W^{1,p}_0(\Omega)}
\leq C(p,\Omega)\|f\|_{L^p(\Omega)}, \ for\ any\ p\in(1, \infty).
\ea\ee

$\bullet$ The function $v=\mathcal{B}[f]$ solves the problem \eqref{2.4}.

$\bullet$ If, moreover, $f$ can be written in the form $f=\div g$ with $g\in [L^r(\Omega)]^3$,
 $g\cdot \vec{n}|_{\partial\Omega}=0$, then
\be\ba\label{bog2}
\|\mathcal{B}[f]\|_{L^r(\Omega)}\leq C(r,\Omega)\|g\|_{L^r(\Omega)},
\ for \ any\ r\in(1, \infty).
\ea\ee
\end{lemma}

For system (\ref{1.8}), the   existence and uniqueness of solutions  were originally established by Matsumura-Padula \cite{MP}.
\begin{lemma}   \la{lem1}{\rm(\cite{MP})}
Let $F \in W^{1,q}(\Omega)$,  for $3 < q \leq \infty$ and suppose that \eqref{mas} holds. Then problem \eqref{1.8} has a unique solution $(\rho_s,0)$,
with $\rho_s\in  W^{1,q}(\om)$ and $\displaystyle \inf_\om \n_s >0$.
% \be\la{1.9}
% \int_\Omega\bigg(\frac{\gamma-1}{\gamma}\big(F-\inf_{\Omega}F\big)\bigg)^{\frac{1}{\gamma-1}}dx<\int_\Omega \rho_0dx.
% \ee
\end{lemma}
Finally, by making slightly modification  to Feireisl-Petzeltov\'a \cite{FP},  we state the following large time behavior
of finite energy weak solutions.
\begin{lemma} \label{lem2}
   Assume  that $\Omega\subset\mathbb{R}^3$  is a bounded domain with a compact and Lipschitz boundary. Let $F$ be Lipschitz continuous on $\overline{\Omega}$ and
    satisfy \eqref{mas}.
     Then, for any finite energy weak solution $(\rho, u)$ to the problem \eqref{1.1}--\eqref{1.4},
      there exists a stationary state $\rho_s$ such that
    \be\la{2.31}
    \rho(t)\rightarrow \rho_s\ strongly \ in \ L^\gamma(\Omega),\ \sup_{\tau>t}\int_\Omega \rho(\tau)|u(\tau)|^2dx\rightarrow 0,
    \ as\ t\rightarrow \infty.
    \ee
    Moreover, the  equilibrium state $(\rho_s, 0)$ is the unique solution to the system \eqref{1.8}.
    \end{lemma}
\begin{remark}\la{rmk2.1}
  In \cite{FP},  $F$   is originally imposed the condition that
  any upper level set
  $ \{x\in\Omega|F(x)>k\}$
  is connected in $\Omega$ for any $k<\displaystyle \sup_{x\in\Omega}F(x)$,
  which ensures  the uniqueness of the system \eqref{1.8}.
  However, in view of  Lemma
  \ref{lem1}, the uniqueness result can be replaced  by the condition \eqref{mas}.
  Thus, we replace the condition by \eqref{mas} in Lemma \ref{lem2}.
\end{remark}

\section{\la{se3}Proof of Theorem \ref{t1}}

First, recalling that $\nabla P(\rho_s)=\rho_s\nabla F$ and by direct calculations, we have
$$
\int_0^\infty \varphi\int_\Omega \rho u \cdot \nabla Fdxdt=\int_0^\infty \varphi_t\int_\Omega\left(\rho_s^\gamma-\frac{\gamma}{\gamma-1}\rho \rho_s^{\gamma-1}\right)dxdt,
$$
which together with \eqref{engyinq} gives
\be\la{x1}
-\int_0^\infty\varphi_t\int_{\Omega}\bigg(\frac{1}{2}\rho|u|^2+G(\rho,\rho_s)\bigg)dxdt+\int_0^\infty\varphi\int\bigg(\mu|\nabla u|^2+(\lam+\mu)(\div{u})^2\bigg)dxdt
\leq 0.
\ee
Moreover, it implies that
\be\ba\la{E11}
\int_\om \left(G(\n,\n_s)+\frac{1}{2}\n|u|^2\right)dx + \int_0^t\int_\Omega [\mu |\na u|^2+(\lambda +\mu)(\div u)^2]dxds
\leq E_0,\ea\ee
for almost every $t\geq 0$.

Next, taking
$\te: = \min\{\frac{2}{3}\ga-1,\frac{1}{2},\frac{\ga}{6}\}$,
 there exists a positive constant $C_0$, depending only on
$\om,\ \gamma,\ F$ and $\ol{\n_0}$, such that, for any $\rho\geq 0$,

\be\la{x2}
C_0^2(\rho^\theta-\rho_s^\theta)^2\leq C_0G(\rho,\rho_s)\leq \rho_s^{-\theta}(\rho^\gamma-\rho_s^\gamma)(\rho^\theta-\rho_s^\theta).
\ee
Then, we claim  the following essential estimate whose proof will be left in the later discussion.

 \begin{pro}\label{3.1}
 There exist positive constants $C$ and $T_0$,
such that, for any $0\leq \varphi(t)\in \mathcal{D}(0,\infty)$ with
supp $\varphi \subset (T_0, \infty)$,
\be\ba\la{3.9}
&
  \frac{C_0}{2}\int_0^\infty \varphi\int_\Omega G(\rho,\rho_s)dxdt\\
&\quad
  +\int_0^\infty \varphi_t \int_\Omega\rho u\cdot\rho_s^{-1}
  \mathcal{B}\left(\rho_s^{1-\te}(\rho^\theta-\rho_s^\theta)
  -\overline{\rho_s^{1-\te}(\rho^\te-\rho_s^\te)}\right)dxdt\\
&
  \leq C\int_0^\infty\varphi\|\nabla u\|^2_{L^2(\Omega)}dt,
\ea\ee
where $C$
depends only on $\om,\ \ga,\ \mu,\ \lam,\ F,\  E_0$ and$\ \ol{\n_0}$,
while $T_0$ depends on
$\om,\ \ga,\ F$ and specially, on  $\n$ and $ u$.
\end{pro}

 With Proposition \ref{3.1} at hand, we are now ready to give the proof of Theorem \ref{t1}.

\noindent{\bf Proof of Theorem \ref{t1}.} Indeed,  adding $(\ref{3.9})$ multiplied by a suitably small constant $\de > 0$,
which will be determined later, to $(\ref{x1})$, we have
\be\ba\label{x24}
  -\int_0^\infty \vp_t V_\de(t) dt + \int_0^\infty \vp W_\de(t) dt \leq 0,
\ea\ee
where
\be\ba\label{v}
  V_\de(t) := \int_\om \left(\frac{1}{2}\rho|u|^2+G(\rho, \rho_s)
  -\de \rho u\cdot\rho_s^{-1}\mathcal{B}(\rho_s^{1-\te}(\rho^\te -\rho_s^\te)
  -\overline{\rho_s^{1-\te}(\rho^\te -\rho_s^\te)}) \right)dx
\ea\ee
and
\be\ba\label{w}
  W_\de(t) := \int_\om \left((\mu - C\de)|\na u|^2
    + (\lambda+\mu)(\div u)^2 + \frac{C_0\de}{2}G(\rho,\rho_s)\right)dx.
\ea\ee
Moreover, it follows from Lemma $\ref{lem2.5}$ and \eqref{rho}  that
\be\ba\label{x25}
&
  \int_\om\rho u\cdot\rho^{-1}_s\mathcal{B}
  \left(\rho_s^{1-\te}(\rho^\theta-\rho_s^\theta)
  -\overline{\rho_s^{1-\te}(\rho^\te-\rho_s^\te)}\right)dx\\
\leq&
  \|\sqrt{\rho} u\|_{L^2(\om)}\|\sqrt{\rho}\|_{L^3(\om)}\|\rho^{-1}_s\|_{L^\infty(\om)}
  \left\lVert\mathcal{B}
  \left(\rho_s^{1-\te}(\rho^\theta-\rho_s^\theta)
  -\overline{\rho_s^{1-\te}(\rho^\te-\rho_s^\te)}\right)\right\rVert_{L^6(\om)}\\
\leq&
  C\|\sqrt{\rho} u\|_{L^2(\om)}
  \left\lVert\na\mathcal{B}
  \left(\rho_s^{1-\te}(\rho^\theta-\rho_s^\theta)
  -\overline{\rho_s^{1-\te}(\rho^\te-\rho_s^\te)}\right)\right\rVert_{L^2(\om)}\\
\leq&
  C\|\sqrt{\rho} u\|_{L^2(\om)}
 \left( \|\rho_s^{1-\te}(\rho^\theta-\rho_s^\theta)\|_{L^2(\om)}+|\overline{\rho_s^{1-\te}(\rho^\te-\rho_s^\te)}|\right)\\
\leq&
  C\|\sqrt{\rho} u\|_{L^2(\om)}
  \|\rho_s^{1-\te}\|_{L^\infty(\om)}\|\rho^\theta-\rho_s^\theta\|_{L^2(\om)}\\
\leq&
  \frac{1}{2}\int_\om \rho|u|^2 dx + C\int_\om G(\rho,\rho_s) dx,
\ea\ee
where in the last inequality we have used $(\ref{x2})$. Then, by choosing  a suitably  small
$\de_0 >0$ which depends only on
$\om,\ \ga,\ \mu,\ \lambda,\ F,\ \ol{\n_0}$ and $E_0$, one has
\be\ba\label{x26}
  \frac{1}{4}\int_\om \left(\frac{1}{2}\rho|u|^2 + G(\rho,\rho_s)\right)dx\leq V_{\de_0}(t)
  \leq 2\int_\om \left(\frac{1}{2}\rho|u|^2 + G(\rho,\rho_s)\right)dx
\ea\ee
and
\be\ba\label{x27}
  W_{\de_0}(t) \geq \frac{1}{4}
  \int_\om (\mu|\na u|^2 + C_0\de_0G(\rho,\rho_s) )dx.
\ea\ee
Note that
\be\ba\label{x28}
  \int_\om \rho|u|^2 dx \leq \|\rho\|_{L^{\frac{3}{2}}(\om)}
    \|u\|_{L^{6}(\om)}^2\leq C(\om,\ga,E_0,\ol{\n_0}) \|\na u\|_{L^2(\om)}^2,
\ea\ee
which combined  with $(\ref{x26})$ and $(\ref{x27})$,  implies that
\be\ba\label{x29}
V_{\de_0(t)}\leq C_1W_{\de_0}(t) \quad \mathrm{a.e.} \ t\in (0,+\infty),
\ea\ee
 where $C_1$ is a positive constant depending
only on $\om,\ \ga,\ \mu,\ \lambda,\ F,\ E_0$ and $\ol{\n_0}$.

Putting  $(\ref{x29})$ into $(\ref{x24})$ gives
for any $0 \leq \vp \in \mathcal{D}(0, \infty)$
with supp $\varphi \subset (T_0,+\infty)$,
\be\ba\label{x30}
  -\int_0^\infty \vp_tV_{\de_0}(t)dt +
  C_1 \int_0^\infty \vp V_{\de_0}(t)dt \leq 0.
\ea\ee
Let $[a,b]$ be any compact subset of $(T_0,+\infty)$ and $\tilde{\eta}$
be the standard mollifier in $\mathbb{R}$.
Taking  $\vp = \ep^{-1} \tilde{\eta}(\frac{\tau -t }{\ep})$
with $\ep \in (0,\min\{a,\frac{b-a}{8}\})$ into $(\ref{x30})$, leads to
\be\ba\label{x31}
\pa_\tau[V_{\de_0}]_\ep + C_1[V_{\de_0}]_\ep \leq 0, \quad
 \mathrm{for\ almost\ every\ }\tau \in[a,b],
\ea\ee
where
\be
[V_{\de_0}]_\ep(\tau):= \frac{1}{\ep}
   \int_0^\infty \tilde{\eta}(\frac{\tau - t}{\ep}) V_{\de_0}(t)dt.
\ee
Thus,  by $(\ref{x31})$, we obtain
\be\ba\label{x32}
  [V_{\de_0}]_\ep(t) \leq [V_{\de_0}]_\ep(s)e^{-C_1(t-s)},
\ea\ee
for almost every $T_0 < s< t <+\infty$.

And then, since $V_{\de_0}(t) \in L^\infty_{\mathrm{loc}}(0,\infty)$,  letting
 $\ep \to 0$ causes
\be\ba\label{x33}
  V_{\de_0}(t) \leq V_{\de_0}(s)e^{-C_1(t-s)}.
\ea\ee
This combined  with $(\ref{x26})$ and $(\ref{E11})$, and letting $s$ converge to $T_0$,
yields   $(\ref{expdec})$.

Finally, we  turn back to prove our  Proposition \ref{3.1} to finish our whole proof.

\noindent{\bf Proof of Proposition \ref{3.1}.}
First, to handle the large  potential external forces,
we introduce the following equality,
initially established  by \cite{HLX}, which plays a crucial role in our subsequent arguments,
\be\ba\la{equation}
&
  \rho_s^{-1}(\nabla(\rho^\gamma-\rho_s^\gamma)
  -\gamma(\rho-\rho_s)\rho_s^{\gamma-2}\nabla\rho_s) \\
&
  =\nabla(\rho_s^{-1}(\rho^\gamma-\rho_s^\gamma))
  -(\gamma-1)G(\rho,\rho_s)\nabla\rho^{-1}_s.
\ea\ee
With the aid of \eqref{equation},
we can rewrite $\eqref{1.1}_2$  as
\be\ba
\nabla(\rho_s^{-1}(\rho^\gamma-\rho_s^\gamma))
  =&-\rho_s^{-1}\left((\rho u)_t+\div{(\rho u\otimes u)}-\mu\Delta u-(\mu+\lam)\na\div{u}\right)
  \\&+(\gamma-1)G(\rho,\rho_s)\nabla\rho^{-1}_s.
\ea\ee

Next, for any $0\leq\vp \in \mathcal{D}(0,\infty)$
with supp $\varphi\in(T_0,+\infty)$,
 we set
\be\la{x4}
\Phi(x,t)=\varphi(t)\rho_s^{-1}\mathcal{B}(\rho_s^{1-\theta}([\rho^\theta]_{\ep}-[\rho_s^\theta]_{\ep})
-\overline{\rho_s^{1-\theta}([\rho^\theta]_{\ep}-[\rho_s^\theta]_{\ep})}),
\ee
where $T_0$ is determined later.

Since $\rho\in L^{\gamma+\theta}(0,T;L^{\ga+\te}(\Omega))$,  we can take $\Phi(x,t)$ as a test function for $(1.1)_2$ to get
\be\la{3.6}\ba
&\int_0^\infty\varphi\int_\Omega\rho_s^{-1}(\rho^\gamma-\rho_s^\gamma)
  (\rho_s^{1-\theta}([\rho^\te]_\ep - [\rho_s^\te]_\ep) - \overline{\rho_s^{1-\theta}([\rho^\te]_\ep - [\rho_s^\te]_\ep)})dxdt\\
&
  +\int_0^\infty\varphi_t\int_\Omega\rho u\cdot\rho_s^{-1}\mathcal{B}(\rho_s^{1-\theta}(\rho^\theta-\rho_s^\theta)
  -\overline{\rho_s^{1-\theta}(\rho^\theta-\rho_s^\theta)})dxdt\\
&
  =-\int_0^\infty\varphi_t\int_\Omega\rho u\cdot\rho_s^{-1}(H_\varepsilon(x,t)-H(x,t))dxdt\\
&\quad
  -\int_0^\infty\varphi\int_\Omega \rho u \cdot\rho_s^{-1} \pa_t H_\varepsilon(x,t) dxdt\\
&\quad
  -\int_0^\infty\varphi\int_\Omega\left(\rho_s^{-1}(\rho u\otimes u):\nabla H_\varepsilon(x,t)
  +\nabla\rho_s^{-1}\cdot(\rho u\otimes u)\cdot H_\varepsilon(x,t)\right)dxdt\\
&\quad
  +\mu\int_0^\infty\varphi\int_\Omega\left(\rho_s^{-1}\nabla u:\nabla H_\varepsilon(x,t)
  +\nabla\rho_s^{-1}\cdot\nabla u\cdot H_\varepsilon(x,t)\right) dxdt\\
&\quad
  +(\lam+\mu)\int_0^\infty\varphi\int_\Omega\rho_s^{-1}{\rm div} u(\rho_s^{1-\theta}([\rho^\theta]_{\ep}-[\rho_s^\theta]_{\ep})
  -\overline{\rho_s^{1-\theta}([\rho^\theta]_{\ep}-[\rho_s^\theta]_{\ep})})dxdt\\
  &\quad
  +(\lam+\mu)\int_0^\infty\varphi\int_\Omega{\rm div} u \nabla\rho_s^{-1}\cdot H_\varepsilon(x,t))dxdt\\
&
  \quad-(\gamma-1)\int_0^\infty\varphi\int_\Omega G(\rho,\rho_s)\nabla\rho_s^{-1}\cdot H_\varepsilon(x,t)dxdt\\
&
  :=-\int_0^\infty\varphi_t I_1dt+\int_0^\infty\varphi\sum_{i=2}^{7}I_idxdt,
\ea\ee
where
\be\ba\la{hh}
&H(x,t) := \mathcal{B}\left(\n_s^{1-\te}(\n^\te - \n_s^\te ) - \ol{\n_s^{1-\te}(\n^\te - \n_s^\te )}\right),\\
&H_\varepsilon(x,t) := \mathcal{B}\left(\n_s^{1-\te}([\n^\te]_\varepsilon - [\n_s^\te]_\varepsilon ) - \ol{\n_s^{1-\te}([\n^\te]_\varepsilon - [\n_s^\te ]_\varepsilon)}\right).
\ea\ee
Then, we estimate each $I_i(i=1,2,\cdots,7)$ as follows.

By virtue of (\ref{bog1}) and \eqref{rho}, we obtain
\be\ba\la{x5}
% -----------------------------------------------------------------
|I_1|&\leq \|\sqrt{\rho}\|_{L^3(\Omega)} \|\sqrt{\rho}u\|_{L^2(\Omega)}\|\rho_s^{-1}\|_{L^\infty(\Omega)}\|H_\varepsilon(x,t) - H(x,t)\|_{L^6(\Omega)}\\
&\leq C\|\nabla(H_\varepsilon(x,t) - H(x,t))\|_{L^2(\Omega)}\\
&
\leq C\|\rho_s^{1-\theta}([\rho^\theta]_{\ep}-[\rho_s^\theta]_{\ep})
-
\rho_s^{1-\theta}(\rho^\theta-\rho_s^\theta)\|_{L^2(\Omega)}
\\&
\leq C\|\rho_s\|^{1-\theta}_{L^\infty(\Omega)}
(\|[\rho^\theta]_\ep-\rho^\theta\|_{L^2(\Omega)}
+\|[\rho_s^\theta]_\ep-\rho_s^\theta\|_{L^2(\Omega)})\\
&\leq  C(\|[\rho^\theta]_\ep-\rho^\theta\|_{L^2(\Omega)}
+\|[\rho_s^\theta]_\ep-\rho_s^\theta\|_{L^2(\Omega)}),
\ea\ee
which together with $\rho^\te \in L^1(0,T;L^2(\om))$ (due to
$2\te <\ga$) gives
\be\ba\label{x34}
\left\lvert\int_0^\infty \vp_t I_1 dt \right\rvert \to 0,\quad
\mathrm{as\ \ep}\to 0,
\ea\ee
and
\be\ba\la{x6}
% ------------------------------------
|I_3|&\leq \|\rho\|_{L^\gamma(\Omega)} \|u\|_{L^6(\Omega)}^2\|\rho_s^{-1}\|_{W^{1,\infty}(\Omega)}
\|H_\varepsilon(x,t)\|_{W^{1,\frac{3\gamma}{2\gamma-3}}(\Omega)}
\\&
\leq C\|\nabla u\|_{L^2(\Omega)}^2\|\rho_s^{1-\te}\|_{L^\infty(\om)}
    \bigg(\|\rho^\theta\|_{L^\frac{3\gamma}{2\gamma-3}(\Omega)}+\|\rho_s^\theta\|_{L^\frac{3\gamma}{2\gamma-3}(\Omega)}\bigg)\\&
\leq C\|\nabla u\|_{L^2(\Omega)}^2,
\ea\ee
where in the last inequality we have used the fact $\theta\leq \frac{2\gamma-3}{3}.$

Moreover, by \eqref{bog1}, \eqref{rho}, we have
\be\ba\la{x7}
|I_4+I_5+I_6|&
\leq C\|\nabla u\|_{L^2(\Omega)}\|\rho_s^{1-\theta}([\rho^\theta]_{\ep}-[\rho_s^\theta]_{\ep})
-\overline{\rho_s^{1-\theta}([\rho^\theta]_{\ep}-[\rho_s^\theta]_{\ep})}\|_{L^2(\Omega)}\\&
\leq C\|\nabla u\|_{L^2(\Omega)}\|\rho_s\|^{1-\theta}_{L^\infty(\Omega)}\|\rho^\theta-\rho_s^\theta\|_{L^2(\Omega)}\\&
\leq C\|\nabla u\|_{L^2(\Omega)}^2
  +\frac{1}{4}\|\rho_s^{-\theta}(\rho^\gamma-\rho_s^\gamma)(\rho^\theta-\rho_s^\theta)\|_{L^1(\Omega)},
\ea\ee
where  in the last inequality,  we have applied $(\ref{x2})$ and Cauchy-Schwarz inequality.

As for $I_7$, by $(\ref{x2})$
\be\ba\la{x8}
% ------------------------------------
|I_7|&\leq C\|G(\rho,\rho_s)\|_{L^1(\Omega)}\|\nabla\rho_s^{-1}\|_{L^\infty(\Omega)}
\|H_\varepsilon(x,t)\|_{L^\infty(\Omega)}\\
&
\leq C\|G(\rho,\rho_s)\|_{L^1(\Omega)}\|[\rho^\theta]_{\ep}-[\rho_s^\theta]_{\ep}
-(\overline{[\rho^\theta]_{\ep}-[\rho_s^\theta]_{\ep}})\|_{L^\frac{\gamma}{\theta}(\Omega)}\\
&
\leq C\|\rho_s^{-\theta}(\rho^\gamma-\rho_s^\gamma)(\rho^\theta-\rho_s^\theta)\|_{L^1(\Omega)}
\|[\rho^\theta]_{\ep}-[\rho_s^\theta]_{\ep}\|_{L^\frac{\gamma}{\theta}(\Omega)}\\
&
\leq C(\om,\ga,F,\ol{\n_0})
\|\rho^\theta-\rho_s^\theta\|_{L^\frac{\gamma}{\theta}(\Omega)}
\cdot
\|\rho_s^{-\theta}(\rho^\gamma-\rho_s^\gamma)(\rho^\theta-\rho_s^\theta)\|_{L^1(\Omega)},
\ea\ee
 due to $\theta\leq \frac{\gamma}{6}.$

We now return to deal with $I_2$.  Multiplying both sides of (\ref{2.3}) by $\rho_s^{1-\te}$ leads to
\be\ba\label{x36}
  (\rho_s^{1-\te}[\rho^\te]_\ep)_t+ \div(\rho_s^{1-\te}[\rho^\te]_\ep u)
  -\na \rho_s^{1-\te}\cdot[\rho^\te]_\ep u
  =(1-\te)\rho_s^{1-\te}[\rho^\te\div u]_\ep + \rho_s^{1-\te}r_\ep.
\ea\ee
Integrating \eqref{x36} over $\Omega$ and by \eqref{1.5} yield
\be\ba\label{x37}
  \overline{(\rho_s^{1-\te}[\rho^\te]_\ep)_t}
  - \overline{\na \rho_s^{1-\te}\cdot [\rho^\te]_\ep u}
  =(1-\te) \overline{\rho_s^{1-\te}[\rho^\te\div u]_\ep}
  +\overline{\rho_s^{1-\te}r_\ep}.
\ea\ee
Combining (\ref{x36}) with (\ref{x37}), we obtain
\be\ba\label{x38}
&\left(\rho_s^{1-\theta}([\rho^\theta]_{\ep}-[\rho_s^\theta]_{\ep})
  -\overline{\rho_s^{1-\theta}([\rho^\theta]_{\ep}-[\rho_s^\theta]_{\ep})}\right)_t\\
=&
  (\rho_s^{1-\te}[\rho^\te]_\ep - \overline{\rho_s^{1-\te}[\rho^\te]_\ep })_t\\
=&
-\div(\rho_s^{1-\te}[\rho^\te]_{\ep} u)
    +\na\rho_s^{1-\te}\cdot[\rho^\te]_\ep u
    -\overline{\na\rho_s^{1-\te}\cdot[\rho^\te]_\ep u}\\
&
  +(1-\te)(\rho_s^{1-\te}[\rho^\te\div u]_\ep
  -\overline{\rho_s^{1-\te}[\rho^\te\div u]_\ep})
  +\rho_s^{1-\te}r_\ep - \overline{\rho_s^{1-\te}r_\ep},
\ea\ee
which gives
\be\ba\la{x10}
I_2&=\int_\Omega \rho u \cdot\rho_s^{-1}\mathcal{B}(\div{(\rho_s^{1-\te}[\rho^\theta]_\ep u)})dx\\
&\quad-(1-\theta)\int_\Omega \rho u \cdot\rho_s^{-1}
\mathcal{B}(\rho_s^{1 - \te}[\rho^\theta\div{u}]_\ep-\overline{\rho_s^{1-\te}[\rho^\theta\div{u}]_\ep})dx\\
&\quad -\int_\Omega \rho u \cdot\rho_s^{-1}\mathcal{B}(\rho_s^{1-\te}r_\ep-\overline{\rho_s^{1-\te}r_\ep})dx\\
&\quad -\int_\Omega \rho u \cdot \rho_s^{-1}\mathcal{B}
\left(\nabla\rho_s^{1-\te}[\rho^\te]_\ep u
- \overline{\nabla\rho_s^{1-\te}[\rho^\te]_\ep u}\right)dx\\
&:=\sum_{i=1}^{4}I_2^i.
\ea\ee

Now we will estimate the above four terms on the right-hand side
of (\ref{x10}) in sequence.

By using (\ref{bog2}), we have
\be\ba\la{x11}
|I_2^1|
&
  \leq C\|\rho\|_{L^{\gamma}(\Omega)}\|u\|_{L^6(\Omega)}\|\mathcal{B}(\div{(\rho_s^{1- \te}[\rho^\theta]_\ep u)})\|_{L^{\frac{6\gamma}{5\gamma-6}}(\Omega)}\\
&
  \leq C \|\na u\|_{L^2(\om)}\|\rho_s^{1-\te}\|_{L^\infty(\om)}
  \|[\rho^\te]_\ep u\|_{L^{\frac{6\gamma}{5\gamma-6}}(\om)}\\
&
  \leq C\|\nabla u\|_{L^2(\Omega)}^2\|\rho^\theta\|_{L^{\frac{3\gamma}{2\gamma-3}}(\Omega)}\\
&
  \leq C\|\nabla u\|_{L^2(\Omega)}^2,
\ea\ee
because of $\theta\leq \frac{2\gamma-3}{3}.$

For $\widetilde{\gamma}:=\min\{\gamma,5\}$,
due to $\te \leq\frac{\ga(2\tilde{\ga}-3)}{3\tilde{\ga}}$, we get by Lemma \ref{lem2.1}(ii) and (\ref{bog1})
\be\ba\la{x12}
|I_2^2|&\leq C\|\rho\|_{L^{\widetilde{\gamma}}(\Omega)}\|u\|_{L^6(\Omega)}
\|\mathcal{B}(\rho_s^{1-\te}[\rho^\theta\div{u}]_\ep-\overline{\rho_s^{1-\te}[\rho^\theta\div{u}]_\ep})
\|_{L^{\frac{6\widetilde{\gamma}}{5\widetilde{\gamma}-6}}(\Omega)}\\&
\leq C\|\nabla u\|_{L^2(\Omega)}\|\nabla\mathcal{B}(\rho_s^{1-\te}[\rho^\theta\div{u}]_\ep-\overline{\rho_s^{1-\te}[\rho^\theta\div{u}]_\ep})
\|_{L^{\frac{6\widetilde{\gamma}}{7\widetilde{\gamma}-6}}(\Omega)}\\&
\leq C\|\nabla u\|_{L^2(\Omega)}\|[\rho^\theta\div{u}]_\ep\|_{L^{\frac{6\widetilde{\gamma}}{7\widetilde{\gamma}-6}}(\Omega)}\\&
\leq C\|\nabla u\|_{L^2(\Omega)}^2\|\rho^\theta\|_{L^{\frac{3\widetilde{\gamma}}{2\widetilde{\gamma}-3}}(\Omega)}\\&
\leq C\|\nabla u\|_{L^2(\Omega)}^2.
\ea\ee

Similarly, by (\ref{bog1}) again, we have
\be\ba\la{x13}
  |I_2^3|&\leq C\|\rho\|_{L^{\widetilde{\gamma}}(\Omega)}\|u\|_{L^6(\Omega)}
  \|\mathcal{B}(\rho_s^{1-\te}r_\ep-\overline{\rho_s^{1-\te}r_\ep})\|
  _{L^{\frac{6\widetilde{\gamma}}{5\widetilde{\gamma}-6}}(\Omega)}\\
&
  \leq C\|\nabla u\|_{L^2(\Omega)}
  \|\nabla\mathcal{B}(\rho_s^{1-\te}r_\ep-\overline{\rho_s^{1-\te}r_\ep})\|
  _{L^{\frac{6\widetilde{\gamma}}{7\widetilde{\gamma}-6}}(\Omega)}\\
&
  \leq C\|\nabla u\|_{L^2(\Omega)}
  \| \rho_s^{1-\te}r_\ep-\overline{\rho_s^{1-\te}r_\ep}\|
  _{L^{\frac{6\widetilde{\gamma}}{7\widetilde{\gamma}-6}}(\Omega)}\\
&
  \leq C \|\nabla u\|_{L^2(\Omega)}\|\rho_s^{1-\te}r_\ep\|
  _{L^{\frac{6\widetilde{\gamma}}{7\widetilde{\gamma}-6}}(\Omega)}\\
&
  \leq C \|\nabla u\|_{L^2(\Omega)}\|r_\ep\|
  _{L^{\frac{6\widetilde{\gamma}}{7\widetilde{\gamma}-6}}(\Omega)},
\ea\ee
which, together with (\ref{E11}) and (\ref{x35}), leads to
\be\ba\label{x20}
  \int_0^\infty \vp |I_2^3| dt \leq
  &
    C\left(\int_0^\infty  \vp\|\na u\|^2_{L^2(\Omega)} dt\right)^\frac{1}{2}
    \cdot \left(\int_0^\infty\vp\|r_\ep\|^2
    _{L^\frac{6\widetilde{\ga}}{7\widetilde{\ga}-6}(\Omega)} dt \right)^\frac{1}{2}\\
  \to& 0, \quad \text{as}\,\, \ep\to 0,
\ea\ee
due to $\te \leq \frac{\ga(2\widetilde{\ga}-3)}{3\widetilde{\ga}}.$

For $I_2^4$,   Lemma \ref{lem2.1}(ii) and (\ref{bog1}) give
\be\ba\la{x14}
  |I_2^4|\leq
&
  C\|\rho\|_{L^\ga(\Omega)} \|u\|_{L^6(\Omega)}
  \|\mathcal{B}(\nabla\rho_s^{1-\te}[\rho^\te]_\ep u
  - \overline{\nabla\rho_s^{1-\te}[\rho^\te]_\ep u})\|_{L^{\frac{6\ga}{5\ga - 6}}(\Omega)}\\
\leq&
  C\|\rho\|_{L^\ga(\Omega)} \|u\|_{L^6(\Omega)}
  \|\na\rho_s^{1-\te}[\rho^{\te}]_\ep u\|_{L^{\frac{6\ga}{5\ga - 6}}(\Omega)}\\
\leq&
  C \|\rho\|_{L^\ga(\Omega)} \|u\|^2_{L^6(\Omega)}
  \|\na\rho_s^{1-\te}\|_{L^\infty(\Omega)}\|\rho^\te\|_{L^{\frac{3\ga}{2\ga -3 }}(\Omega)}\\
\leq&
  C\|\na u \|^2_{L^2(\Omega)},
  \ea\ee
where we have used the fact $\theta\leq \frac{2\gamma-3}{3}.$

Next, for the first term on the lefthand side of \eqref{3.6}, we have
\be\ba\label{x15}
&
  \int_0^\infty\varphi\int_\Omega\rho_s^{-1}(\rho^\gamma-\rho_s^\gamma)
  \left(\rho_s^{1-\theta}([\rho^\te]_\ep - [\rho_s^\te]_\ep) - \overline{\rho_s^{1-\theta}([\rho^\te]_\ep - [\rho_s^\te]_\ep)}\right)dxdt\\&
=\int_0^\infty\vp\int_\Omega\rho_s^{-1}(\rho^\gamma-\rho_s^\gamma)
  \left(\rho_s^{1-\te}(\rho^\te - \rho_s^\te) - \overline{\rho_s^{1-\te}(\rho^\te - \rho_s^\te)}\right) dxdt\\
&
  \quad +\int_0^\infty\varphi\int_\Omega\rho_s^{-1}(\rho^\gamma-\rho_s^\gamma)\left( \rho_s^{1-\te}([\rho^\te]_{\ep} -\rho^\te) - \rho_s^{1-\te}([\rho_s^\te]_\ep - \rho_s^\te)\right)
  dxdt\\
&
  \quad - \int_0^\infty\varphi\int_\Omega\rho_s^{-1}(\rho^\gamma-\rho_s^\gamma)
  (\overline{\rho_s^{1-\te}([\rho^\te]_\ep - \rho^\te) -\rho_s^{1-\te} ([\rho^\te_s]_\ep - \rho^\te_s)}) dxdt\\
&
  :=\int_0^\infty\vp\int_\Omega\rho_s^{-\te}(\rho^\gamma-\rho_s^\gamma)
  (\rho^\theta-\rho_s^\theta)dxdt + K_1+K_2+K_3,
\ea\ee
where
\be\ba\label{k1}
K_1 :=
&\int_0^\infty\varphi\int_\Omega\rho_s^{-1}(\rho^\gamma-\rho_s^\gamma)\left( \rho_s^{1-\te}([\rho^\te ]_{\ep} -\rho^\te) - \rho_s^{1-\te}([\rho_s^\te]_\ep - \rho_s^\te)\right)
dxdt,\\
\ea\ee
\be\ba\la{k2}
K_2 : =   - \int_0^\infty\varphi\int_\Omega\rho_s^{-1}(\rho^\gamma-\rho_s^\gamma)
(\overline{\rho_s^{1-\te}([\rho^\te]_\ep - \rho^\te) -\rho_s^{1-\te} ([\rho^\te_s]_\ep - \rho^\te_s)}) dxdt,\\
\ea\ee
\be\ba\la{k3}
K_3 := -\int_0^\infty\vp\int_\Omega \rho_s^{-1}(\rho^\ga - \rho_s^\ga)
 \overline{\rho_s^{1-\te}(\rho^\te - \rho_s^\te)}dxdt.
\ea\ee
As for $K_1$, Lemma \ref{lem2.1}(ii) and H\"older's inequality imply that
\be\ba\label{kk1}
|K_1|
&
  \leq C\|\rho_s^{-\te}\|_{L^{\infty}(\Omega)}
  \int_0^\infty\vp
  \|\rho^\gamma-\rho_s^\gamma\|_{L^{\frac{\gamma+\theta}{\ga}}(\Omega)}
  \|[\rho^\theta]_\ep-\rho^\theta\|_{L^{\frac{\gamma+\theta}{\theta}}(\Omega)}dt\\
 &\quad+C\|\rho_s^{-\te}\|_{L^{\infty}(\Omega)}
  \int_0^\infty\vp
  \|\rho^\gamma-\rho_s^\gamma\|_{L^{\frac{\gamma+\theta}{\ga}}(\Omega)}
  \|[\rho_s^\theta]_\ep-\rho_s^\theta\|_{L^{\frac{\gamma+\theta}{\theta}}(\Omega)}dt\\
 &
  \leq C\bigg(\int_0^\infty\vp\|\rho^\gamma-\rho_s^\gamma\|_{L^{\frac{\gamma+\theta}{\ga}}(\Omega)}^{\frac{\gamma+\theta}{\ga}}dt\bigg)^{\frac{\ga}{\ga+\theta}} \bigg(\int_0^\infty\vp\|[\rho^\theta]_\ep-\rho^\theta\|_{L^{\frac{\gamma+\theta}{\theta}}(\Omega)}^{\frac{\gamma+\theta}{\theta}}dt\bigg)^{\frac{\theta}{\ga+\theta}}\\
&
  \quad+C(\vp)\bigg(\int_0^\infty\vp\|\rho^\gamma-\rho_s^\gamma\|_{L^{\frac{\gamma+\theta}{\ga}}(\Omega)}^{\frac{\gamma+\theta}{\ga}}dt\bigg)^{\frac{\ga}{\ga+\theta}}
  \|[\rho_s^\theta]_\ep -\rho_s^\theta\|_{L^{\frac{\gamma+\theta}{\theta}}(\Omega)}\\
&
  \leq C\bigg(\int_0^\infty\vp\|[\rho^\theta]_\ep-\rho^\theta\|_{L^{\frac{\gamma+\theta}{\theta}}(\Omega)}
    ^{\frac{\gamma+\theta}{\theta}}dt\bigg)^{\frac{\theta}{\ga+\theta}}
  +C\|[\rho_s^\theta]_\ep-\rho_s^\theta\|_{L^{\frac{\gamma+\theta}{\theta}}(\Omega)}\\
&\to 0, \ \mathrm{as} \ \ep \to 0
\ea\ee
because of $\rho \in L^{\ga+\te}(0,T;L^{\ga+\te}(\Omega))$
  and $ \ \rho_s \in L^{\ga + \te}(\Omega).$

Similarly,
\be\ba\la{kk2}
|K_2|&\leq C
  \int_0^\infty
   \vp\left(\int_\Omega \left\lvert \rho^\ga - \rho_s^\ga\right\rvert dx
  \cdot \int_\Omega \big(| [\rho^\te]_\ep - \rho^\te| + | [\rho_s^\te]_\ep - \rho_s^\te| \big)dx\right) dt\\
&\leq
C\int_0^\infty \vp
    \left(\|[\rho^\te]_\ep - \rho^\te \|_{L^\frac{\ga+\te}{\te}(\Omega)}
      +\|[\rho_s^\te]_\ep - \rho_s^\te \|_{L^\frac{\ga+\te}{\te}(\Omega)}\right)dt\\
&\to 0, \ \mathrm{as} \ \ep \to 0
\ea\ee
due to
 $ \rho^\te \in L^\frac{\ga+\te}{\te}(0,T; L^{\frac{\ga+\te}{\te}}(\om))$
  and $ \rho_s^\te \in L^\frac{\ga + \te}{\te}(\Omega).$

For  $K_3$,
integrating  $(\ref{texp})$ over $\Omega$
and applying  $(\ref{1.8})_4$ lead to
\be\ba\label{x16}
&
\left\lvert  \int_\Omega\rho_s^{1-\te}(\rho^\te - \rho_s^\te)dx \right\rvert\\&
 = \left\lvert \left(\frac{1}{\te}-1\right)\int_\om  \left((\rho^\te - \rho_s^\te)^2
  \int_0^1(1-\tau)
  \left(\rho_s^\te+\tau(\rho^\te - \rho^\te_s)\right)^{\frac{1}{\te} - 2}d\tau \right)dx\right\rvert\\
&
  \leq C\int_\Omega (\rho^\te - \rho_s^\te)^2
     \left(\|\rho_s\|^{1-2\te}_{L^\infty(\Omega)}
      + |\rho^\te - \rho_s^\te|^{\frac{1-2\te}{\te}}\right)dx\\
&
  \leq C \int_\Omega \left((\rho^\te - \rho_s^\te)^2 + |\rho^\te - \rho^\te_s|^\frac{1}{\te}\right)dx\\
&
  \leq C \int_\Omega G(\rho,\rho_s ) dx,
\ea\ee
where in the last inequality we have used the fact
\be\ba\label{x17}
(\rho^\te -\rho_s^\te)^2 + |\rho^\te - \rho_s^\te|^\frac{1}{\te}
  \leq C(\om,\ga,F,\ol{\n_0})G(\rho, \rho_s),
\ea\ee
due to $\te \leq\min\{\frac{\ga}{6},\frac{1}{2}\}$.

Substituting $(\ref{x16})$ into $(\ref{k3})$ induces that
\be\ba\label{x18}
|K_3|
&
  \leq C\left\lvert\int_0^\infty \vp
  \left( \int_\Omega \rho_s^{-1} (\rho^\ga - \rho_s^\ga) dx
  \int_\Omega G(\rho, \rho_s) dx \right)dt\right\rvert\\
&
 \leq C \sup_{T_0<t<\infty}\int_\om \left\lvert\rho^\ga -\rho_s^\ga\right\rvert dx
\cdot
  \int_0^\infty \vp
 \int_\Omega G(\rho, \rho_s) dx dt\\
&
\leq C(\om,\ga,F,\ol{\n_0},E_0) \sup_{T_0<t<\infty}
  \|\rho-\rho_s\|_{L^\ga(\om)} \cdot \int_0^\infty \vp\int_\Omega \rho_s^{-\theta}(\rho^\ga-\rho_s^\ga)(\rho^\theta-\rho_s^\theta)dxdt,
\ea\ee
where in the last inequality we have used
\be\ba\label{x39}
\int_\om \left|\rho^\ga -\rho_s^\ga\right| dx
&
 \leq\int_\om\int_0^1
  \left\lvert\ga(\rho-\rho_s)\left((1-\xi)\rho_s +\xi\rho\right)^{\ga -1}\right\rvert d\xi
  dx\\
&
  \leq C\|\rho-\rho_s\|_{L^\ga(\om)}
  \int_0^1\left\lVert((1-\xi)\rho_s +\xi\rho)^{\ga -1}
    \right\rVert_{L^\frac{\ga}{\ga-1}(\om)}d\xi\\
&
  \leq C\|\rho-\rho_s\|_{L^\ga(\om)}
  (\|\rho\|_{L^{\ga}(\om)}^{\ga-1}
  +\|\rho_s\|_{L^{\ga}(\om)}^{\ga-1})\\
&
  \leq C \|\rho-\rho_s\|_{L^\ga(\om)}.
\ea\ee

Combining  $(\ref{3.6})$, \eqref{x5}--\eqref{x15}, \eqref{kk1}--\eqref{kk2}
and $(\ref{x18})$ gives
\be\ba\label{x19}
&
  \int_0^\infty\vp\int_\Omega\rho_s^{-\te}(\rho^\gamma-\rho_s^\gamma)
    (\rho^\theta-\rho_s^\theta)dxdt\\
&\quad
  +\int_0^\infty\vp_t\int_\Omega\rho u\cdot\rho_s^{-1}
  \mathcal{B}\left(\rho_s^{1-\te}(\rho^\theta-\rho^\te_s)
  -\overline{\rho_s^{1-\te}(\rho^\te-\rho_s^\te)}\right)dxdt\\
&\leq
  \left(\frac{1}{4}+C\sup_{T_0<t<\infty}(\|\rho^\te - \rho_s^\te\|_{L^\frac{\gamma}{\theta}(\Omega)}
  + \|\rho-\rho_s\|_{L^\ga(\om)})\right)
  \int_0^\infty\vp\int_\Omega \rho_s^{-\te}(\rho^\ga - \rho_s^\ga)
  (\rho^\te - \rho_s^\te)dxdt\\
&\quad +
  C\int_0^\infty \vp \|\na u \|_{L^2(\Omega)}^2dt+
  \left\lvert\int_0^\infty \vp_t I_1 dt \right\rvert
  +\int_0^\infty \vp|I_2^3| dt + |K_1| + |K_2|\\
& :=
  \left(\frac{1}{4}+C\sup_{T_0<t<\infty}(\|\rho^\te - \rho_s^\te\|_{L^\frac{\gamma}{\theta}(\Omega)}
  + \|\rho-\rho_s\|_{L^\ga(\om)})\right)
  \int_0^\infty\vp\int_\Omega \rho_s^{-\te}(\rho^\ga - \rho_s^\ga)
  (\rho^\te - \rho_s^\te)dxdt\\
&\quad +
  C\int_0^\infty \vp \|\na u \|_{L^2(\Omega)}^2dt + M(\ep).
\ea\ee
Here, it follows from  $(\ref{x34})$, $(\ref{x20})$, $(\ref{kk1})-(\ref{kk2})$ and  $\rho^\te \in L^1(0,T;L^2(\om)),\,
\rho_s^\te \in L^2(\om)$ that  $\displaystyle \lim_{\ep \to 0}M(\ep)=0$.

Then taking  $\ep \to 0$ in both sides of ($\ref{x19}$), we obtain
\be\ba\label{x21}
&
  \int_0^\infty\vp\int_\Omega\rho_s^{-\te}(\rho^\gamma-\rho_s^\gamma)
    (\rho^\theta-\rho_s^\theta)dxdt\\
&\quad
  +\int_0^\infty\vp_t\int_\Omega\rho u\cdot\rho_s^{-1}
  \mathcal{B}\left(\rho_s^{1-\te}(\rho^\theta-\rho^\theta_s)
  -\overline{\rho_s^{1-\te}(\rho^\te-\rho_s^\te)}\right)dxdt\\
&\leq
  \left(\frac{1}{4}+\tilde{C}\sup_{T_0<t<\infty}(\|\rho^\te - \rho_s^\te\|_{L^\frac{\gamma}{\theta}(\Omega)}
  + \|\rho-\rho_s\|_{L^\ga(\om)})\right)
  \int_0^\infty\vp\int_\Omega \rho_s^{-\te}(\rho^\ga - \rho_s^\ga)
  (\rho^\te - \rho_s^\te)dxdt\\
&\quad +
  C\int_0^\infty \vp \|\na u \|_{L^2(\Omega)}^2dt,
\ea\ee
where $\tilde{C}$ is a positive constant depending on
$\om,\ \ga,\ F,\ \ol{\n_0}$ and $E_0$.

By Lemma \ref{lem2}, we select a sufficiently large constant
$T_0$, depending on $\om,\ \ga,\ F,\ \n$ and $u$,
such that
\be\la{xx2}
\sup_{T_0<t<\infty}\left(\|\rho - \rho_s\|_{L^\ga(\Omega)}
+\norm[L^\ga(\om)]{\n-\n_s}^\te\right)
\leq \frac{1}{8\tilde{C}},
\ee
which together with
$$|\rho^\te - \rho_s^\te|^{\frac{\gamma}{\theta}}\leq |\rho - \rho_s|^\ga$$
gives
\be\ba\label{x22}
  \sup_{T_0<t<\infty} \|\rho^\te- \rho_s^\te\|_{L^{\frac{\gamma}{\theta}}(\Omega)}
  \leq \sup_{T_0<t<\infty}\|\rho - \rho_s\|_{L^\ga(\Omega)}^\te
  \leq \frac{1}{8\tilde{C}}.
\ea\ee

In the end, combining $(\ref{x21})$--$(\ref{x22})$ with $(\ref{x2})$,
we arrive at $(\ref{3.9})$ which completes the proof of Proposition \ref{3.1}.

\section{\la{se4}Further discussion on assumption (\ref{mas})}

In this section, we provide further discussion on the assumption (\ref{mas}).

 Based on the arguments in Remarks \ref{rmk2} and  \ref{rmk2.1},
 the significance of (\ref{mas}) lies in its guarantee of
 the existence and uniqueness of the solution $(\rho_s, 0)$
 to the system (\ref{1.8}) with
 $0 < \rho_s \in W^{1,\infty}(\om)$. A natural question is  whether
 uniqueness of $(\rho_s, 0)$ and the property $0 < \rho_s \in W^{1,\infty}(\om)$ still hold
 if (\ref{mas}) does not hold, namely,
 \be\ba\label{x40}
 \int_\Omega\bigg(\frac{\gamma-1}{\gamma}\big(F-\inf_{\Omega}{F}\big)\bigg)^{\frac{1}{\gamma-1}}dx
 \geq \int_\Omega \rho_0dx.
 \ea\ee

To answer this question,
we present a useful lemma established by Radek \cite{E} as follows, which generalizes the
existence and uniqueness results of \eqref{1.9} from Feireisl-Petzeltov\'a \cite{FP0}.
\begin{lemma}\la{lem4}
  Let $\om$ be a domain and $\ga\in(1,\infty)$  Assume that   $F$ is a Lipschitz continuous
  function on $\displaystyle \overline{\om}$ and $\displaystyle -\infty<\inf_\om F \leq \sup_\om F <\infty$. We set
  \be\ba\la{mk3}
  \hat{B}:=\{k\in(-\infty, \sup_\om F)|
  \{x\in\om: F(x)>k\}\ \mathrm{is\ not\ connected}\}
  \ea\ee
  and
  \be\ba\la{mk4}
  \hat{K} = \begin{cases}
      \inf \hat{B}\quad \,\,\,\ &\mathrm{if} \ \hat{B} \neq\varnothing,\\
      \sup\limits_{\om} F \quad \,\ &\mathrm{if} \ \hat{B} = \varnothing,
  \end{cases}
  \ea\ee
  \be\ba\la{6}
  m:=\int_\om \rho dx\in (0, \infty), \ \ \hat{m}: = \int_\om \left(\frac{\ga-1}{\ga}( F-\hat{K})_+  \right)^{\frac{1}{\ga-1}}dx,
  \ea\ee
  then

{\rm(i)}  if $m \geq \hat{m}$, then there is at most one solution $\rho\in L^\infty_{loc}(\om)$ satisfying \eqref{1.9}.

{\rm(ii)} if $0< \hat{m}< \infty ,\ 0 < m< \hat{m},$ then there is
  a continuum of solutions to \eqref{1.9} in $L^\infty_{loc}(\om)$.
\end{lemma}

With Lemma \ref{lem4} at hand, we are now in a position to
point out the necessity of  the assumption
$(\ref{mas})$.
\begin{lemma}\la{lem5}
  Let $F\in W^{1,\infty}(\om)$. Assume that \eqref{x40} holds, then
  there exists at least a solution $(\rho_s, 0)$ to the system \eqref{1.8}
  with $\rho_s \in L^\infty(\om)$
  and moreover, any solution $(\rho_s,0)$
  with $\n_s\in L^\infty_{\text{loc}}(\om)$ satisfies that $\rho_s $ contains
  vacuum on $\ol{\om}$.
\end{lemma}

\begin{proof} Define a function $S : (- \infty, \displaystyle\sup_{\om} F]
  \to [0, \infty)$ by
  \be\ba\label{rs3}
    S(k) = \int_\om \left(\frac{\ga-1}{\ga}(F-k)_+\right)
    ^\frac{1}{\ga - 1} dx.
  \ea\ee
Then $S(k)$ is a decreasing continuous function
on $(-\infty, \displaystyle\sup_\om F]$,
and
\be\ba\la{mk1}
0 = S(\sup_\om F)< \int_\om \rho_s dx=\int_\Omega \rho_0 dx \leq S(\inf_\om F) = \int_\om \left(\frac{\ga-1}{\ga}( F-\inf_\om F)  \right)^{\frac{1}{\ga-1}}dx.
\ea\ee
This implies that there exists a constant $k_0$ satisfying $\displaystyle \inf_\om F \leq k_0 <
\sup_\om F$, such that  $(\rho_s,0)$ is a solution to system (\ref{1.8}), where
\be\ba\label{mk2}
\rho_s:=\left(\frac{\gamma-1}{\gamma }(F-k_0)_+\right)^{\frac{1}{\gamma-1}}dx.
\ea\ee
Especially, it is easy to see $\displaystyle k_0 = \inf_\om F$ if
\be\ba\label{mk5}
\int_\Omega\bigg(\frac{\gamma-1}{\gamma}\big(F-\inf_{\Omega}{F}\big)\bigg)^{\frac{1}{\gamma-1}}dx
= \int_\Omega \rho_0dx.
\ea\ee
That is, there exists a solution $(\rho_s,0)$ with
$\rho_s \in L^\infty(\om)$
containing vacuum on $\overline{\om}$ to system (\ref{1.8})
and moreover, if $\frac{3}{2}< \ga \leq 2$, then $\rho_s$ belongs to $W^{1,\infty}(\om)$.

Next, we state that under the assumption (\ref{x40}), if $(\rho_s,0)$ is a solution to (\ref{1.8}) with
$\rho_s \in L^\infty_{loc}(\om)$, then $\rho_s$ must contain vacuum on $\overline{\om}$.
In this sense, condition (\ref{mas}) is therefore of indispensability in this paper.
To this end, we consider the following two  cases.

Case 1.  If
\be\ba\label{mk6}
\int_\Omega\bigg(\frac{\gamma-1}{\gamma}\big(F-\inf_{\Omega}{F}\big)\bigg)^{\frac{1}{\ga-1}}dx
= \int_\Omega \rho_0dx,
\ea\ee
then
\be\ba\la{mk7}
  \int_\om \rho_0 dx
  &= \int_\om\left(\frac{\ga-1}{\ga}(F-\inf_\om F)\right)^\frac{1}{\ga-1}dx\\
  &\geq  \int_\om \left(\frac{\ga-1}{\ga}( F-\hat{K})_+  \right)^{\frac{1}{\ga-1}}dx =\hat{m} ,
\ea\ee
by  the fact $\displaystyle \inf_{\om} F \leq \hat{K}$ due to the
definition of $\hat{B}$.
Hence by Lemma \ref{lem4} (\romannumeral1), it holds that $(\rho_s, 0)$
with
\be\la{mk8}
  \rho_s = \left(\frac{\ga-1}{\ga}( F-\inf_\om F)  \right)^{\frac{1}{\ga-1}}
\ee
is the unique solution to (\ref{1.8}). Obviously, $\rho_s$ contains vacuum on $\ol{\om}$.

Case 2.  If
\be\ba\label{mk9}
\int_\Omega\bigg(\frac{\gamma-1}{\gamma}\big(F-\inf_{\Omega}{F}\big)\bigg)^{\frac{1}{\ga-1}}dx
> \int_\Omega \rho_0dx,
\ea\ee
and contradictorily assume that there exists a solution $(\rho_s, 0)$ to (\ref{1.8})
with $\rho_s > 0$ on $\ol{\om}$, {then $\rho_s$  satisfies}
\be\la{mk10}
  \rho_s = \left(\frac{\ga-1}{\ga}( F-k)  \right)^{\frac{1}{\ga-1}}
\ee
with $k < \displaystyle\inf_\om F$. Hence it implies that
\be\ba\la{mk11}
\int_\om \rho_s dx = \int_\om\left(\frac{\ga-1}{\ga}( F-k)  \right)^{\frac{1}{\ga-1}}dx > \int_\om \left(\frac{\ga-1}{\ga}( F-\inf_\om F)  \right)^{\frac{1}{\ga-1}}dx,
\ea\ee
which is in contradictory to  (\ref{mk9}). So far,  we finish the proof of Lemma \ref{lem5}.
\end{proof}


\begin {thebibliography} {99}

\bibitem{AF} {\sc R.A. Adams, J.J.F. Fournier}, {\em
Sobolev Spaces}, 2nd ed., Pure and Applied Mathematics(Amsterdam), vol.140, Elsevier/Academic Press, Amsterdam, 2003.

\bibitem{B} {\sc M.E. Bogovskii}, {\em
Solutions of some problems of vector analysis, associated with the operators \div and grad},
 Theory of Cubature Formulas and the Application of Functional Analysis to Problems
of Mathematical Physics, Trudy Sem. S.L. Soboleva, vol. 1, Akad. Nauk SSSR Sibirsk. Otdel.,
Inst. Mat. Novosibirsk, 1980, 5--40, 149 (Russian).

\bibitem{CHS}{\sc G.C. Cai, B. Huang, X.D. Shi}, {\em On compressible Navier-Stokes equations subject to large
potential forces with slip boundary conditions in 3D
bounded domains}(2021),
arXiv: 2101.12572v1.

\bibitem{DL} {\sc R.J. Diperna, P.L. Lions}, {\em Ordinary differential equations, transport theory and
Sobolev spaces},
 Invent. Math, {\bf 98} (1989), no. 3, 511--547.

\bibitem{E} {\sc R. Erban}, {\em On the static-limit solutions to the Navier-Stokes equations of compressible
flow},
J. Math. Fluid Mech. 3 (2001), no. 4, 393–408.

\bibitem{F} {\sc E. Feireisl}, {\em
Dynamics of Viscous Compressible Fluids}, Oxford Lecture Series in Mathematics and
its Applications vol. 26, Oxford University Press, Oxford, 2004.

\bibitem{FNP} {\sc E. Feireisl, A. Novotn\'{y}, H. Petzeltov\'{a}}, {\em
On the existence of globally defined weak solutions to the Navier-Stokes equations},
 J. Math. Fluid Mech, {\bf 3}(2001), 358--392.

\bibitem{FP0} {\sc E. Feireisl, H. Petzeltov\'a}, {\em On the zero-velocity-limit solutions to the Navier-Stokes equations of compressible flow},
Manuscripta Math. {\bf 97} (1998), no.~1, 109--116.

\bibitem{FP} {\sc E. Feireisl, H. Petzeltov\'{a}}, {\em Large-time behaviour of solutions to the Navier-Stokes equations of compressible flow},
Arch. Ration. Mech. Anal. {\bf 150} (1999) 77--96.


\bibitem{FZZ}{\sc D.Y. Fang, R.Z. Zi, T. Zhang}, {\em Decay estimates for isentropic compressible Navier-Stokes
equations in bounded domain},
J. Math. Anal. Appl. {\bf 386} (2012), no. 2, 939--947.

\bibitem{G} {\sc G.P. Galdi}, {\em An Introduction to the Mathematical Theory of the Navier-Stokes Equations. Vol. I:
Linearized Steady Problems}, Springer Tracts in Natural Philosophy, vol. 38, Springer-Verlag, New
York, 1994.

\bibitem{H1}{\sc D. Hoff}, {\em Global solutions of the Navier-Stokes equations for multidimensional compressible flow with discontinuous initial data},
J. Differ. Equ. {\bf 120} (1995), no.~1, 215--254.

\bibitem{H2}{\sc D. Hoff}, {\em Strong convergence to global solutions for multidimensional flows of compressible, viscous fluids with polytropic equations of state and discontinuous initial data},
Arch. Ration. Mech. Anal. {\bf 132} (1995), no.~1, 1--14.

\bibitem{HLX}{\sc F. Huang, J. Li, Z. Xin}, {\em Convergence to equilibria and blowup behavior of global strong solutions to the Stokes approximation equations
for two-dimensional compressible flows with large data}, J. Math. Pures Appl. {\bf 86} (9) (2006), no.~6, 471--491.

\bibitem{JZ}{\sc S. Jiang, P. Zhang}, {\em On spherically symmetric solutions of the compressible isentropic Navier-Stokes equations},
 Comm. Math. Phys. {\bf 215} (2006), no.~3, 559--581.

\bibitem{L1} {\sc P.L. Lions}, {\em Mathematical Topics in Fluid Mechanics. Vol. 1: Incompressible Models},
Oxford Lecture Series in Mathematics and its Applications, vol. 3, The Clarendon Press, Oxford University
Press, New York, 1996. Oxford Science Publications.

\bibitem{L2} {\sc P.L. Lions}, {\em Mathematical Topics in Fluid Mechanics. Vol. 2: Compressible Models},
Oxford Lecture Series in Mathematics and its Applications, vol. 10, The Clarendon Press, Oxford University
Press, New York, 1996. Oxford Science Publications.


\bibitem{LZZ}{\sc J. Li, J.W. Zhang, J.N. Zhao}, {\em On the global motion of viscous compressible
barotropic flows subject to large external potential forces and vacuum},
SIAM J. Math. Anal. {\bf 47} (2015), 1121--1153.

\bibitem{MN1}{\sc A. Matsumura, T. Nishida}, {\em The initial value problem for the equations of motion
of viscous and heat-conductive gases},
J. Math. Kyoto Univ. {\bf 20}(1) (1980), 67--104.

\bibitem{MN2}{\sc A. Matsumura, T. Nishida}, {\em Initial boundary value problems for the equations of
motion of general fluids},
in: Proc. Fifth Int. Symp. on Computing Methods in
Appl. Sci. and Engng. (1982), 389--406.

\bibitem{MN3}{\sc A. Matsumura, T. Nishida}, {\em Initial value problem for the equations of motion of
compressible viscous and heat-conductive fluids},
Comm. Math. Phys. {\bf 89} (1983), 445--464.

\bibitem{MP} {\sc A. Matsumura, M. Padula}, {\em Stability of the stationary solutions of compressible viscous fluids with large external forces},
 Stab. Appl. Anal. Cont. Media {\bf 2} (1992) 183--202.

\bibitem{N}{\sc J. Nash}, {\em Le probl\`{e}me de Cauchy pour les \'{e}quations diff\'{e}rentielles d'un fluide g\'{e}n\'{e}ral},
 Bull. Soc. Math. France {\bf 90} (1962), 487--497 (French).


\bibitem{NS}{\sc A. Novotn\'{y}, I. Str\v{a}skraba}, {\em Stabilization of weak solutions to compressible Navier-Stokes equations},
J. Math. Kyoto Univ. {\bf 40} (2) (2000), 217--245.

\bibitem{P}{\sc M. Padula}, {\em On the exponential stability of the rest state of a viscous compressible
fluid}, J. Math. Fluid Mech. {\bf 1} (1999), no. 1, 62--77.

\bibitem{PSW} {\sc Y.F. Peng, X.D. Shi, Y.S. Wu}, {\em Exponential decay for Lions-Feireisl's weak solutions to the barotropic com
pressible Navier-Stokes equations in 3D bounded domains}, Indiana U. Math. J. {\bf 70} (2021), no. 5, 1813--1831.



\bibitem{Se}{\sc J. Serrin}, {\em On the uniqueness of compressible fluid motion},
 Arch. Ration.  Mech. Anal. {\bf 3} (1959), 271--288.

\bibitem{So}{\sc V.A. Solonnikov}, {\em Solvability of initial boundary value problem for the equation
of motion of viscous compressible fluid},
 Steklov Inst. Seminars in Math. Leningrad {\bf 56} (1976), 128--142.

\bibitem{ST}{\sc Y. Shibata, K. Tanaka}, {\em On the steady flow of compressible viscous fluid and its
stability with respect to initial disturbance},
J. Math. Soc. Japan {\bf 55} (3) (2003), 797--826.

\bibitem{T}{\sc A. Tani}, {\em On the first initial-boundary value problem of compressible viscous fluid},
motion. Publ. RIMS Kyoto Univ. {\bf 13} (1977), 193--253.


\end {thebibliography}
\end{document}